\documentclass[12pt,a4paper]{article}
\usepackage{indentfirst}
\usepackage{amsmath,amssymb,amsfonts,amsthm,graphics}
\usepackage{makeidx}
\usepackage{ifpdf}
\newtheorem{thm}{Theorem}

\newtheorem{lem}{Lemma}

\theoremstyle{definition}

\def\-{\mbox{--}}

\newtheorem{pro}{Proposition}

\newtheorem{obs}{Observation}

\newtheorem{definition}{Definition}

\begin{document}
\title{\bf Graphs with $4$-rainbow index $3$ and $n-1$\Large\bf \footnote{Supported
by NSFC Nos. 11371205 and 11071130, and the ``973" program.}}
\author{\small Xueliang~Li$^1$, Ingo Schiermeyer$^2$, Kang Yang$^1$, Yan~Zhao$^1$\\
\small $^1$Center for Combinatorics and LPMC-TJKLC\\
\small Nankai University, Tianjin 300071, China\\
\small  lxl@nankai.edu.cn;
 yangkang@mail.nankai.edu.cn; zhaoyan2010@mail.nankai.edu.cn\\
\small $^2$Institut f\"{u}t Diskrete Mathematik und Algebra\\
\small Technische Universit\"{a}t Bergakademie Freiberg\\
\small 09596 Freiberg, Germany. Ingo.Schiermeyer@tu-freiberg.de }
\date{}
\maketitle
\begin{abstract}

Let $G$ be a nontrivial connected graph with an edge-coloring
$c:E(G)\rightarrow \{1,2,\ldots,q\},$ $q\in \mathbb{N}$, where
adjacent edges may be colored the same. A tree $T$ in $G$ is called
a $rainbow~tree$ if no two edges of $T$ receive the same color. For
a vertex set $S\subseteq V(G)$, a tree that connects $S$ in $G$ is
called an {\it $S$-tree}. The minimum number of colors that are
needed in an edge-coloring of $G$ such that there is a rainbow
$S$-tree for every $k$-set $S$ of $V(G)$ is called the {\it
$k$-rainbow index} of $G$, denoted by $rx_k(G)$. Notice that an
lower bound and an upper bound of the $k$-rainbow index of a graph
with order $n$ is $k-1$ and $n-1$, respectively. Chartrand et al.
got that the $k$-rainbow index of a tree with order $n$ is $n-1$ and
the $k$-rainbow index of a unicyclic graph with order $n$ is $n-1$
or $n-2$. Li and Sun raised the open problem of characterizing the
graphs of order $n$ with $rx_k(G)=n-1$ for $k\geq 3$. In early
papers we characterized the graphs of order $n$ with 3-rainbow index
2 and $n-1$. In this paper, we focus on $k=4$, and characterize the
graphs of order $n$ with 4-rainbow index 3 and $n-1$, respectively.

{\flushleft\bf Keywords}: rainbow $S$-tree, $k$-rainbow index.

{\flushleft\bf AMS subject classification 2010}: 05C05, 05C15, 05C75.

\end{abstract}

\section{Introduction}

All graphs considered in this paper are simple, finite and
undirected. We follow the terminology and notation of Bondy and
Murty \cite{Bondy}. Let $G$ be a nontrivial connected graph with an
edge-coloring $c: E(G)\rightarrow \{1,2,\ldots,q\}$, $q\in
\mathbb{N}$, where adjacent edges may be colored the same. A path of
$G$ is a \emph{rainbow path} if any two edges of the path have
distinct colors. $G$ is \emph{rainbow connected} if any two vertices
of $G$ are connected by a rainbow path. The minimum number of colors
required to make $G$ rainbow connected is called its \emph{rainbow
connection number}, denoted by $rc(G)$. Results on the rainbow
connectivity can be found in \cite{Cai, Caro, Chartrand1,
ChartrandZhang, Chartrand, Lishisun, Lisun}.

These concepts were introduced by Chartrand et al. in
\cite{Chartrand1}. In \cite{Zhang}, they generalized the concept of
rainbow path to rainbow tree. A tree $T$ in $G$ is called a
$rainbow~tree$ if no two edges of $T$ receive the same color. For
$S\subseteq V (G)$, a $rainbow\ S$-$tree$ is a rainbow tree that
connects $S$. Given a fixed integer $k$ with $2\leq k \leq n$, the
edge-coloring $c$ of $G$ is called a $k$-$rainbow~coloring$ of $G$
if for every set $S$ of $k$ vertices of $G$, there exists a rainbow
$S$-tree, and we say that $G$ is $k$-$rainbow~connected$. The
$k$-$rainbow~index$ $rx_k(G)$ of $G$ is the minimum number of colors
that are needed in a $k$-$rainbow~coloring$ of $G$. Clearly, when
$k=2$, $rx_2(G)$ is nothing new but the rainbow connection number
$rc(G)$ of $G$. For every connected graph $G$ of order $n$, it is
easy to see that $rx_2(G)\leq rx_3(G)\leq \cdots \leq rx_n(G)$.

The $Steiner~distance$ $d_G(S)$ of a set $S$ of vertices in $G$ is
the minimum size of a tree in $G$ that connects $S$. Such a tree is
called a $Steiner~S$-$tree$ or simply an $S$-$tree$. The
$k$-$Steiner~diameter$ $sdiam_k(G)$ of $G$ is the maximum Steiner
distance of $S$ among all sets $S$ with $k$ vertices in $G$. Then
there is a simple upper bound and lower bound for $rx_k(G)$.

\begin{obs}[\cite{Zhang}]\label{lowerupperbound}
For every connected graph $G$ of order $n\geq 3$ and each integer
$k$ with $3\leq k\leq n$, we have $k-1\leq sdiam_k(G)\leq
rx_k(G)\leq n-1$.
\end{obs}
It is easy to get the following observations.

\begin{obs}[\cite{Zhang}]\label{cutedge}
Let $G$ be a connected graph of order $n$ containing two bridges $e$ and $f$. For each
integer $k$ with $2\leq k\leq n$, every $k$-rainbow coloring of $G$ must assign distinct
colors to $e$ and $f$.
\end{obs}

\begin{obs}[\cite{CLYZ}]\label{subgraph}
Let $G$ be a connected graph of order $n$, and $H$ be a connected spanning subgraph of
$G$. Then $rx_k(G)\leq rx_k(H)$.
\end{obs}

The following is an immediate consequence of the observations above.
Namely, trees attain the upper bound of $k$-rainbow index,
regardless of the value of $k$.

\begin{pro}[\cite{Zhang}]\label{tree}
Let $T$ be a tree of order $n\geq 3$. For each integer $k$ with $3\leq k\leq n$, $rx_k(T)=n-1$.
\end{pro}

In \cite{Zhang}, they also showed that the $k$-rainbow index of a unicyclic graph is $n-1$ or $n-2$.

\begin{thm}[\cite{Zhang}]\label{unicyclic}
If $G$ is a unicyclic graph of order $n\geq 3$ and girth $g\geq 3$, then
\begin{equation}
 rx_k(G)=
   \begin{cases}
      n-2, & \text{$k=3$ and $g\geq4$}; \\
     n-1, & \text{$g=3$ or $4\leq k\leq n$}.
    \end{cases}
\end{equation}
\end{thm}

Notice that an lower bound and an upper bound of the $k$-rainbow
index of a graph with order $n$ is $k-1$ and $n-1$, respectively. In
\cite{Lisun}, the authors raised an open problem: for $k\geq 3$,
characterize the graphs of order $n$ with $rx_k(G)=n-1$. It is not
easy to settle down the problem for general $k$. In \cite{CLYZ} and
\cite{LYZ}, we characterized the graphs of order $n$ with 3-rainbow
index 2 and $n-1$, respectively. In this paper we mainly deal with
the 4-rainbow index of graphs with order $n$. More specifically,
characterize the graphs of order $n$ whose 4-rainbow index is 3 and
$n-1$, respectively.

\section{Characterize the graphs with $rx_4(G)\bf{=3}$}

First we give a necessary and sufficient condition for $rx_4(G)=3$.
Note that if a connected graph of order 4 has three colors, then it
has a rainbow spanning tree. Thus, the following lemma holds.

\begin{lem}\label{4vertices}
Let $G$ be a connected graph of order $n$ ($n\geq4$). Then
$rx_4(G)=3$ if and only if each induced subgraph of $G$ with order
$4$ is connected and has three different colors.
\end{lem}

Next we give some necessary conditions for $rx_4(G)=3$. By Lemma
\ref{4vertices}, it is easy to get the following proposition.

\begin{pro}\label{Delta}
Let $G$ be a graph of order $n$ with $rx_4(G)=3$, where $n\geq5$.
Then $\delta(G)\geq n-3$ and $\Delta(\overline{G})\leq 2$. In other
words, $\overline{G}$ is the union of some paths (may be trivial)
and cycles.
\end{pro}

For fixed integers $p$, $q$, an edge-coloring of a complete graph
$K_n$ is called a $(p, q)$-$coloring$ if the edges of every
$K_p\subseteq K_n$ are colored with at least $q$ distinct colors.
Clearly, $(p, 2)$-colorings are the classical Ramsey colorings
without monochromatic $K_p$ as subgraphs. Let $f(n,p,q)$ be the
minimum number of colors needed for a $(p,q)$-coloring of $K_n$. In
\cite{erdos}, Erd$\ddot{o}$s got that $f(10,4,3)=4$, and so the
following proposition holds.

\begin{pro}\label{dianshu}
Let $G$ be a graph of order $n$ with $rx_4(G)=3$. Then $n\leq9$.
\end{pro}

By Lemma \ref{4vertices} and Theorem \ref{unicyclic}, we get the following proposition.

\begin{pro}\label{c4c5}
Let $G$ be a connected graph of order $n$ ($n\geq4$) with
$rx_4(G)=3$. Then $\overline{G}$ contains neither $C_4$ nor $C_5$.
\end{pro}

When $G$ is a graph of order 4, it is obvious that $rx_4(G)=3$ if
and only if $G$ is connected. Now we divide into five cases by the
value of $n$ with $5\leq n\leq 9$.

\begin{lem}\label{n=5}
Let $G$ be a connected graph of order 5. Then $rx_4(G)=3$ if and
only if $\overline{G}$ is a subgraph of $P_5$ or $K_2\cup K_3$.
\end{lem}

\begin{proof}
Let $G$ be a graph with $rx_4(G)=3$. By Proposition \ref{Delta}, it
is easy to check that if $\overline{G}$ is not a subgraph of $P_5$
or $K_2\cup K_3$, then $\overline{G}$ is isomorphic to $C_4$ or
$C_5$, a contradiction by Proposition \ref{c4c5}.

Conversely, by Observation \ref{subgraph}, we need to provide an
edge-coloring $C: E\rightarrow \{1,2,3\}$ of $G$ when $\overline{G}$
is isomorphic to $P_5$ or $K_2\cup K_3$. Suppose $\overline{G}$ is
isomorphic to $P_5$, denote $V(\overline{G})=\{v_1,\cdots,v_5\}$ and
$E(\overline{G})=\{v_1v_2,v_2v_3,v_3v_4,v_4v_5\}$. Set
$c(v_1v_3)=2$, $c(v_1v_4)=1$, $c(v_1v_5)=3$, $c(v_2v_4)=3$,
$c(v_2v_5)=2$, $c(v_3v_5)=1$. Suppose $\overline{G}$ is isomorphic
to $K_2\cup K_3$, denote $V(\overline{G})=\{v_1,\cdots,v_5\}$ and
$E(\overline{G})=\{v_1v_2,v_2v_3,v_1v_3,v_4v_5\}$. Set
$c(v_1v_4)=1$, $c(v_1v_5)=2$, $c(v_2v_4)=2$, $c(v_2v_5)=3$,
$c(v_3v_4)=3$, $c(v_3v_5)=1$. It is easy to show that the two
edge-colorings make $G$ 4-rainbow connected.
\end{proof}

\begin{lem}\label{n=6}
Let $G$ be a graph of order 6. Then $rx_4(G)=3$ if and only if
$\overline{G}$ is a subgraph of $C_6$ or $2K_3$.
\end{lem}
\begin{proof}
Let $G$ be a graph with $rx_4(G)=3$. By Proposition \ref{Delta}, if
$\overline{G}$ is not a subgraph of $C_6$ or $2K_3$, then
$\overline{G}$ contains $C_4$ or $C_5$, a contradiction by
Proposition \ref{c4c5}.

Conversely, by Observation \ref{subgraph}, we need to provide an
edge-coloring $C: E\rightarrow \{1,2,3\}$ of $G$ when $\overline{G}$
is isomorphic to $C_6$ or $2K_3$. Suppose $\overline{G}$ is
isomorphic to $C_6$, denote $V(\overline{G})=\{v_1,\cdots,v_6\}$ and
$E(\overline{G})=\{v_1v_3,v_1v_4,v_1v_5,v_2v_4,v_2v_5,v_2v_6,v_3v_5,v_3v_6,v_4v_6\}$.
Set $c(v_1v_3)=2$, $c(v_1v_4)=3$, $c(v_1v_5)=1$, $c(v_2v_4)=1$,
$c(v_2v_5)=2$, $c(v_2v_6)=3$, $c(v_3v_5)=3$, $c(v_3v_6)=1$,
$c(v_4v_6)=2$. Suppose $\overline{G}$ is isomorphic to $2K_3$,
denote $V(\overline{G})=\{v_1,\cdots,v_6\}$ and
$E(\overline{G})=\{v_1v_4,v_1v_5,v_1v_6,v_2v_4,v_2v_5,v_2v_6,v_3v_4,v_3v_5,v_3v_6\}$.
Set $c(v_1v_4)=3$, $c(v_1v_5)=2$, $c(v_1v_6)=1$, $c(v_2v_4)=1$,
$c(v_2v_5)=3$, $c(v_2v_6)=2$, $c(v_3v_4)=2$, $c(v_3v_5)=1$,
$c(v_3v_6)=3$. It is easy to show that the two edge-colorings make
$G$ 4-rainbow connected.
\end{proof}

It is a tedious work to check whether a graph is 4-rainbow connected
when $7 \leq n\leq 9 $. Hence we introduce an algorithm with the
idea of backtracking to deal with such cases. We should point out
that the algorithm has a good performance when $n\leq 9$, although
the time complexity is not polynomial.
\begin{tabbing}
\noindent\rule[0.25\baselineskip]{\textwidth}{2pt}
\\\textbf{Algorithm} 4-rainbow Coloring of a graph\\
\noindent\rule[0.25\baselineskip]{\textwidth}{1pt}\\
Input: a graph $G=(V,E)$ with $V=\{v_1,v_2,...,v_n\}$, $E=\{e_1,e_2,...,e_m\}$.\\
Output: give a 4-rainbow coloring $colorlist[m]$ of $G$,
or verify that $G$ has no 4-rainbow \\coloring.\\
1. reorder the edge sequence $e_1,e_2,...,e_m$, to
make sure $E(G[v_1,...,v_t]) =\{e_1,...,e_s\}$,\\$~~$ where
$s$ denotes the number of edges of $G[v_1,...,v_t]$, where $1\leq t\leq n$.\\
2. fix the color of $e_1$ with 1. Initialize $i=2$ and $colorlist=[1,0,0,...,0]$; \\
3. while $i\geq2$\\
$~~~~~~$if $i>m$\\
$~~~~~~~~~~~~$show $colorlist$; stop;\\
$~~~~~~~~~~~~$$colorlist[i]=colorlist[i]+1$;\\
$~~~~~~$if $colorlist[i]>3$\\
$~~~~~~~~~~~~$$colorlist[i]=0$; $i--$;\\
$~~~~~~$else if \textbf{Boolean CHECK}($e_i$)\\
$~~~~~~~~~~~~$$i++$;\\
4. there is no 4-rainbow coloring; stop.\\

\\Boolean CHECK($e_s$)\\
Input: a graph $G=(V,E)$ with $V=\{v_1,v_2,...,v_n\}$,
$E=\{e_1,e_2,...,e_m\}$ with the order \\described above.
Set $e_s=(v_p,v_q)$, where $p<q$. Give a coloring of
the first $s$ edges of $E(G)$. \\
Output: determine whether the given coloring is illegal.\\
1. for $i=1$ up to $q-2$ and $i\neq p$\\
$~~~~~~$for $j=i+1$ up to $q-1$ and $j\neq p$\\
$~~~~~~~~~~~~$if all edges of the induced subgraph
$G[v_i,v_j,v_p,v_q]$ are colored but $G[v_i,v_j,v_p,v_q]$ is
\\$~~~~~~~~~~~$ not 4-rainbow colored. \\
$~~~~~~~~~~~~~~~~~$return $false$; stop;\\
2. return $true$; stop.\\
\noindent\rule[0.25\baselineskip]{\textwidth}{1pt}
\end{tabbing}

\begin{lem}\label{n=7}
Let $G$ be a graph of order 7. Then $rx_4(G)=3$ if and only if
$\overline{G}$ is a subgraph of $C_6$ or $2K_2\cup K_3$ or $P_5\cup
K_2$ or $2K_3$.
\end{lem}
\begin{proof}
Let $G$ be a graph with $rx_4(G)=3$. By Proposition \ref{Delta}, if
$\overline{G}$ is not a subgraph of $C_6$ or $2K_2\cup K_3$ or
$P_5\cup K_2$ or $2K_3$, then by Proposition \ref{c4c5},
$\overline{G}$ is isomorphic to $P_4\cup P_3$ or $P_4\cup K_3$ or
$P_7$ or $C_7$. By Observation \ref{subgraph}, we need only to
verify that $rx_4(G)\neq3$ when $\overline{G}$ is isomorphic to
$P_4\cup P_3$, by the algorithm, $rx_4(G)\neq3$.

Conversely, by Observation \ref{subgraph} again, we need to provide
an edge-coloring of $G$ when $\overline{G}$ is isomorphic to $C_6$
or $2K_2\cup K_3$ or $P_5\cup K_2$ or $2K_3$. The four colorings are
shown in Figure 1. It is easy to show that these four colorings make
$G$ 4-rainbow connected.
\end{proof}

\begin{figure}[h,t,b,p]
\begin{center}
\scalebox{0.7}[0.7]{\includegraphics{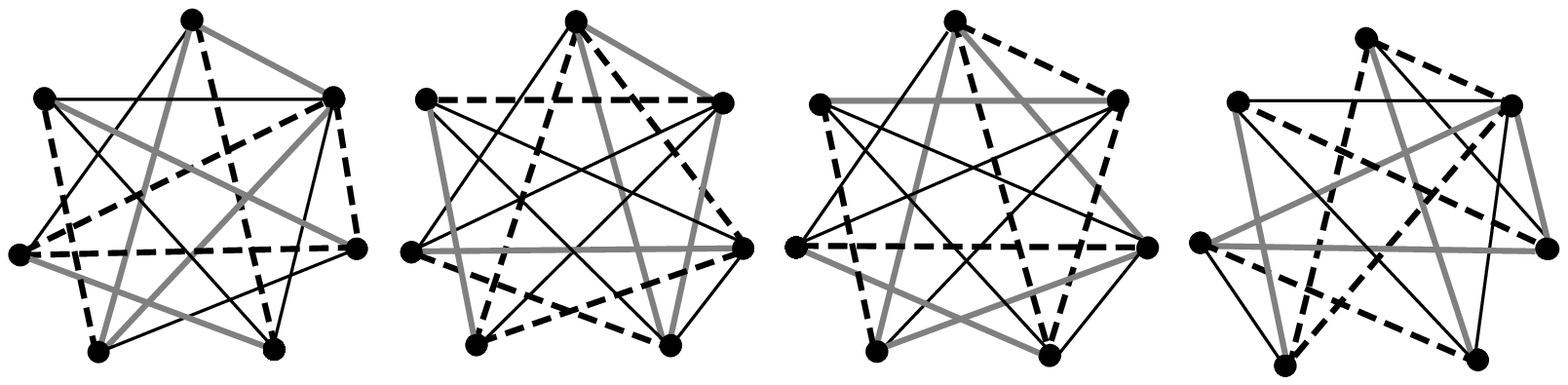}}\\[15pt]

Figure~1. Graphs for Lemma \ref{n=7} (the same type of lines stand
for the same color)
\end{center}
\end{figure}

\begin{lem}\label{n=8}
Let $G$ be a graph of order 8. Then $rx_4(G)=3$ if and only if
$\overline{G}$ is a subgraph of $K_2\cup 2K_3$ or $P_6\cup K_2$.
\end{lem}

\begin{proof}
Let $G$ be a graph with $rx_4(G)=3$. By Proposition \ref{Delta}, if
$\overline{G}$ is not a subgraph of $K_2\cup 2K_3$ or $P_6\cup K_2$,
then by Proposition \ref{c4c5}, it is easy to check that either
$\overline{G}$ contains $P_4\cup P_3\cup K_1$ or $\overline{G}$ is
isomorphic to $C_6\cup 2K_1$. By Observation \ref{subgraph}, we need
to verify that $rx_4(G)\neq3$ when $\overline{G}$ is isomorphic to
$P_4\cup P_3\cup K_1$ or $\overline{G}$ is isomorphic to $C_6\cup
2K_1$. If $\overline{G}$ is isomorphic to $P_4\cup P_3\cup K_1$,
then by Lemma \ref{n=7}, $rx_4(G)\neq3$. If $\overline{G}$ is
isomorphic to $C_6\cup 2K_1$, by the algorithm, $rx_4(G)\neq3$.

Conversely, by Observation \ref{subgraph} again, we need to provide
an edge-coloring of $G$ when $\overline{G}$ is isomorphic to
$K_2\cup 2K_3$ or $P_6\cup K_2$. The two edge-colorings are shown in
Figure 2. It is easy to show that the two edge-colorings make $G$
4-rainbow connected.
\end{proof}

\begin{figure}[h,t,b,p]
\begin{center}
\scalebox{0.7}[0.7]{\includegraphics{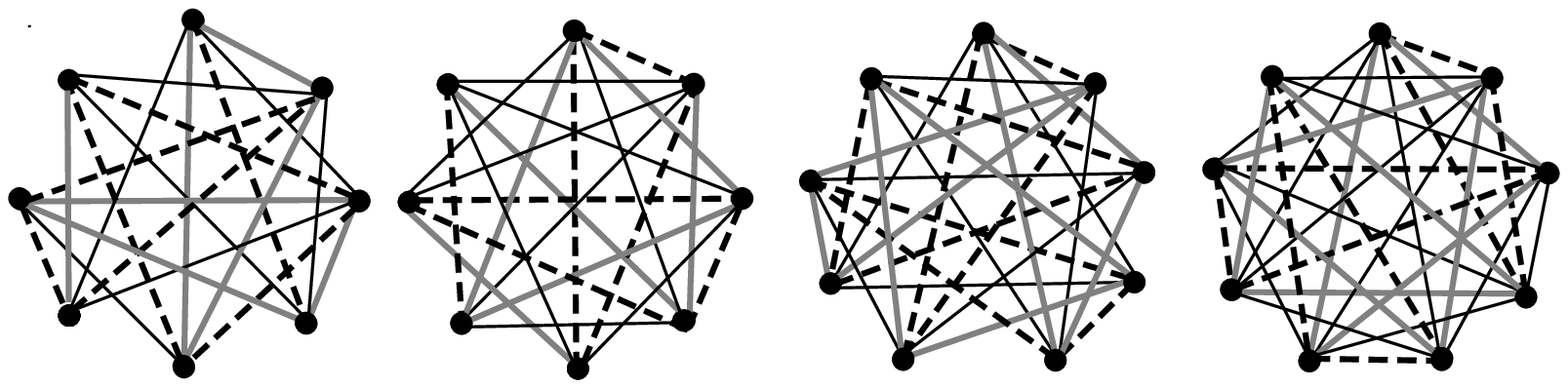}}\\[15pt]

Figure~2. Graphs for Lemma \ref{n=8}, \ref{n=9} (the same type of
lines stand for the same color)
\end{center}
\end{figure}

\begin{lem}\label{n=9}
Let $G$ be a graph of order 9. Then $rx_4(G)=3$ if and only if
$\overline{G}$ is a subgraph of $3K_3$ or $P_3\cup 3K_2$.
\end{lem}
\begin{proof}
Let $G$ be a graph with $rx_4(G)=3$. By Proposition \ref{Delta}, if
$\overline{G}$ is not a subgraph of $3K_3$ or $P_3\cup 3K_2$, then
by Proposition \ref{c4c5}, it is easy to check that either
$\overline{G}$ contains $P_4$ or $\overline{G}$ is isomorphic to
$K_3\cup 3K_2$. By Observation \ref{subgraph}, we need to verify
that $rx_4(G)\neq3$ when $\overline{G}$ is isomorphic to $P_4$ or
$K_3\cup 3K_2$, by the algorithm, in each case, $rx_4(G)\neq3$.

Conversely, by Observation \ref{subgraph} again, we need only to
provide an edge-coloring of $G$ when $\overline{G}$ is isomorphic to
$3K_3$ or $P_3\cup 3K_2$. The two edge-colorings are shown in Figure
2. It is easy to show that the two edge-colorings make $G$ 4-rainbow
connected.
\end{proof}

Combining the preceding five lemmas, we are ready to characterize
the graphs whose 4-rainbow index is 3.

\begin{thm}\label{thm1}
$rx_4(G)=3$ if and only if $G$ is one of the following graphs: (1)
$G$ is a connected graph of order 4; (2) $G$ is of order 5 and
$\overline{G}$ is a subgraph of $P_5$ or $K_2\cup K_3$; (3) $G$ is
of order 6 and $\overline{G}$ is a subgraph of $C_6$ or $2K_3$; (4)
$G$ is of order 7 and $\overline{G}$ is a subgraph of $C_6$ or
$2K_2\cup K_3$ or $P_5\cup K_2$ or $2K_3$; (5) $G$ is of order 8 and
$\overline{G}$ is a subgraph of $K_2\cup 2K_3$ or $P_6\cup K_2$; (6)
$G$ is of order 9 and $\overline{G}$ is a subgraph of $3K_3$ or
$P_3\cup 3K_2$.
\end{thm}

\section{Characterize the graphs with $rx_4(G)\bf{=n-1}$}

First of all, we need some notation and basic results.

\begin{definition}
Let $G$ be a connected graph with $n$ vertices and $m$ edges. Define
the $cyclomatic~number$ of $G$ as $c(G)=m-n+1$. A graph $G$ with
$c(G)=k$ is called a $k$-$cyclic$ graph. According to this
definition, if a graph $G$ meets $c(G)=0$, 1, 2 or 3, then $G$ is
called acyclic (or a tree), unicyclic, bicyclic, or tricyclic,
respectively.
\end{definition}

\begin{definition}
For a subgraph $H$ of $G$ and $v\in V(G)$, let $d(v,H)=min\{d_G(v,x): x\in V(H)\}$.
\end{definition}

Let $G$ be a connected graph. To $contract$ an edge $e=uv$ is to
delete $e$ and replace its ends by a single vertex incident to all
the edges which were incident to either $u$ or $v$. Let $G^{'}$ be
the graph obtained by contracting some edges of $G$. Given a rainbow
coloring of $G^{'}$, when it comes back to $G$, we can extend $G'$
back to $G$, we keep the colors of the corresponding edges of
$G^{'}$ in $G$ and assign a fresh color to a new edge. More
specifically, if one new vertex is added to the current graph $G'$,
give the new incident edge joining to $G'$ a fresh color. Then $G$
can be made to be 4-rainbow connected. Hence, the following lemma
holds.

\begin{lem}\label{contract}
Let $G$ be a connected graph, and $G^{'}$ be a connected graph by contracting some
edges of $G$. Then $rx_4(G)\leq rx_4(G^{'})+|V(G)|-|V(G^{'})|$.
\end{lem}

The $\Theta$-$graph$ is a graph consisting of three internally disjoint paths with
common end vertices and of lengths $a$, $b$, and $c$, respectively, such that
$a\leq b\leq c$. It follows that if a $\Theta$-graph has order $n$, then $a+b+c=n+1$.

Let $G$ be a connected graph of order $n$, to $subdivide$ an edge
$e$ is to delete $e$, add a new vertex $x$, and join $x$ to the ends
of $e$. We will first give some sufficient conditions to make sure
that the 4-rainbow index of $G$ never attains the upper bound $n-1$.

\begin{figure}[h,t,b,p]
\begin{center}
\scalebox{1}[1]{\includegraphics{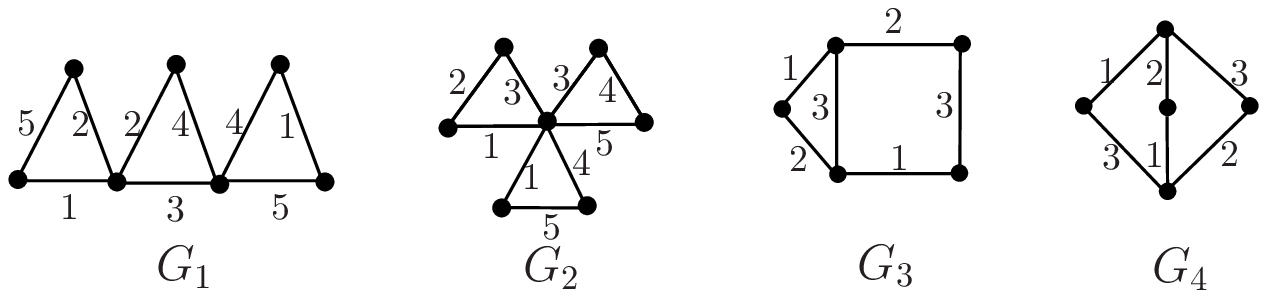}}\\[15pt]

Figure~3. Graphs for Lemma \ref{theta}
\end{center}
\end{figure}

\begin{lem}\label{theta}
Let $G$ be a connected graph of order $n$. If $G$ contains three edge-disjoint cycles,
or a $\Theta$-graph of order at least 5 as subgraphs, then $rx_4(G)\leq n-2$.
\end{lem}
\begin{proof}
Consider two graphs $G_1$, $G_2$ in Figure 3, and by checking the
given edge-coloring in the figure, we have $rx_4(G_i)\leq
|V(G_i)|-2$, $i=1,2$. Then if $G$ contains three edge-disjoint
cycles $C_1,C_2,C_3$, we can extend the three triangles of $G_1$ or
$G_2$ to $C_1,C_2$ and $C_3$ respectively by a sequence of
operations of subdivision. Then add to the cycles an additional set
of edges, to get a spanning subgraph $G'$ of $G$. By Observation
\ref{subgraph} and Lemma \ref{contract}, we have $rx_4(G)\leq
rx_4(G')\leq rx_4(G_i)+|V(G^{'})|-|V(G_i)|\leq n-2$.

Let $\mathcal{G}$ be the set of $\Theta$-graphs whose order is
exactly 5. Then $\mathcal{G}=\{G_3,G_4\}$ (see Figure 3). By
checking the given edge-coloring, we have $rx_4(G_i)\leq
|V(G_i)|-2$, $i=3,4$. Similarly, $rx_4(G)\leq n-2$ follows.
\end{proof}

A graph $G$ is a $cactus$ if every edge is part of at most one cycle in $G$.

\begin{lem}\label{lem3}
Let $G$ be a cactus of order $n$ and  $c(G)=2$. Then $rx_4(G)=n-1$.
\end{lem}
\begin{proof}
Let the two cycles of $G$ be $C^1$ and $C^2$, where
$C^1=v_1v_2\cdots v_{\ell}v_1$, $C^2=v'_1v'_2\cdots v'_{\ell'}v'_1$,
the unique path connecting the two cycles be $v_iPv'_j$, where the
two end-vertices $v_i$ and $v'_j$ may coincide. Suppose we have a
color set $C$ and $|C|=n-2$. Set $C=\{1,2,\cdots,n-2\}$ and $E_i$ is
the set of edges colored with $i$, $c_i=|E_i|$, $1\leq i\leq n-2$.
Without loss of generality, we always set $c_1\geq c_2\geq\cdots\geq
c_{n-2}$. Notice that $\sum^{n-2}_{i=1}c_i=n+1$. We divide into the
following cases.

{\bf Case 1.}~~ $c_1=4$, $c_2=c_3=\cdots =c_{n-2}=1$. We have the following claim.

{\bf Claim 1.}~~ No three edges of $C^1$ or $C^2$ have the same color.

\emph{Proof.}~Suppose $c(v_1v_2)=c(v_pv_{p+1})=c(v_qv_{q+1})$, where $v_1v_2$,
$v_pv_{p+1}$, $v_qv_{q+1}$ are three distinct edges. Let $S=\{v_1,v_p,v_q\}$. It
is easy to check that any tree connecting $S$ contains at least two edges of $v_1v_2$,
$v_pv_{p+1}$ and $v_qv_{q+1}$, this contradiction proves the claim.

By Observation \ref{cutedge} and Claim 1, at least 3 edges of $E_1$
exist on cycles and each cycle has at most two of them. Suppose
$v_1v_2$ and $v_pv_{p+1}$ of $C^1$ have color 1, we distinguish two
subcases: (1) there is a cut edge $uu'$ in $E_1$. Suppose
$d(u,C^1)\geq d(u',C^1)$ and $d(u,v_i)=d(u,C^1)$, where $2\leq i\leq
p$. Any tree connecting $v_1$ and $u$ contains at least two edges
colored with 1. (2) no cut edge has color 1. Then at least two
edges, say $v'_1v'_2$ and $v'_qv'_{q+1}$ of $C^2$ have color 1, and
the end-vertices of the path connecting $C^1$ and $C^2$ are $v_i$
and $v'_j$, where $2\leq i\leq p$, $2\leq j\leq q$. Again, any tree
connecting $v_1$ and $v'_1$ contains at least two edges in $E_1$.

{\bf Case 2.}~~ $c_1=3,c_2=2,c_3=\cdots =c_{n-2}=1$. We also have the following claim.

{\bf Claim 2.}~~ No four edges of a cycle can have only two colors.

\emph{Proof.}~Suppose otherwise four edges, $v_1v_2$, $v_pv_{p+1}$,
$v_qv_{q+1}$, $v_rv_{r+1}$ of $C^1$ have color $a$ or $b$, where
$a,b\in C$. Set $S=\{v_1,v_p,v_q,v_r\}$. It is easy to check that
any tree connecting S contains at least three of the four edges
above. By the Pigeon Hole Principle, one of the two colors occurs at
least twice, a contradiction.

By Claim 2, at most three edges of $C^i(i=1,2)$ can have colors 1
and 2. Notice that $|E_1\cup E_2|=5$. Since no two cut edges can
have the same color, there are the following possibilities: (1)
three edges of $E_1\cup E_2$ are in a cycle, say $C^1$. Then there
exist cut edges in $E_1\cup E_2$, or the other two edges of $E_1\cup
E_2$ are both in $C^2$. Similar to Case 1, we can choose three
vertices such that no rainbow tree connects them. (2) two edges of
$E_1\cup E_2$ are in each cycle. Then a cut edge $uu'$ exists in
$E_1\cup E_2$. There are two situations according to the positions
of $uu'$ and the other four edges of $E_1\cup E_2$ in cycles. We can
always find three vertices such that any tree connecting them
contains at least three edges of $E_1\cup E_2$. (3) two edges of
$E_1\cup E_2$ are in one cycle, and other two of them are cut edges.
The argument is similar, and it also produces a contradiction.

{\bf Case 3.}~~ $c_1=c_2=c_3=2,c_4= \cdots =c_{n-2}=1$. In a number
of subcases similar to those in Cases 1 and 2, a set $S$ of vertices
can be found such that a tree connecting them contains at least four
edges from $E_1\cup E_2\cup E_3$. So by the Pigeon Hole Principle
again, one of the three colors occurs at least twice.

By the analysis above, all the possibilities of an $(n-2)$-coloring
lead to a contradiction, thus we have $rx_4(G)\geq n-1$. On the
other hand, by Observation \ref{lowerupperbound}, it follows that
$rx_4(G)= n-1$.
\end{proof}

To characterize all the graphs with 4-rainbow index $n-1$, we need
to introduce more graphs. Let $\mathcal{G}_1$ be the set of graphs
by identifying each vertex of $K_4$ with an end-vertex of an
arbitrary path, and $\mathcal{G}_2$ be the set of graphs by
identifying each vertex of $K_4-e$ with the root of an arbitrary
tree.

\begin{lem}\label{lem4}
Let $G$ be a connected graph of order $n$. If $G\in
\mathcal{G}_1\cup \mathcal{G}_2$, then $rx_4(G)=n-1$.
\end{lem}

\begin{proof}
Suppose $G\in \mathcal{G}_1$, and $v_1$, $v_2$, $v_3$ and $v_4$ are
the four pendant vertices of $G$. We have
$d_{G}(v_1,v_2,v_3,v_4)=n-1$. Combining with Observation
\ref{lowerupperbound}, we have $rx_4(G)=n-1$. Let $G\in
\mathcal{G}_2$. Denote by $H$ the induced subgraph $K_4-e$ of $G$,
where $E(H)=\{v_1v_2,v_2v_3,v_3v_4,v_4v_1,v_2v_4\}$ and denote by
$T_i$ the tree rooted at $v_i$, $i=1,2,3,4$. We have the following
claim.

{\bf Claim 3.}~~ No three edges of $H$ share colors with the cut edges.

\emph{Proof.}~Let $v'_iv''_i,1\leq i\leq 3$, be the cut edges whose
colors exist in $H$. We may assume that $d(v'_i,H)\geq d(v''_i,H)$.
Notice that the deletion of any three edges of $H$ disconnects $G$,
and we will get some components. Let $v$ be an arbitrary vertex of
$H$ in the component different from the one containing $v'_1$. Set
$S=\{v,v'_1,v'_2,v'_3\}$. There is no rainbow tree connecting $S$,
which verifies Claim 3.

Now we are aiming to prove that $H$ needs at least three fresh
colors different from the $n-4$ colors of cut edges to make sure
that $G$ is 4-rainbow connected. Then we get the conclusion
$rx_4(G)=n-1$. Since $rx_4(H)=3$ and by Claim 3, one or two edges of
$H$ have the color of cut edges. Assume first that the colors of cut
edges $v'_1v''_1$, $v'_2v''_2$ appear in $H$. Suppose $d(v'_i,H)\geq
d(v''_i,H)$, $i=1,2$. Since the deletion of two edges incident to a
vertex of degree two disconnects $H$, the position of the two edges
of $H$ having the colors of cut edges may have the following
possibilities: $v_1v_4$, $v_2v_4$ or $v_1v_4$, $v_3v_4$ or $v_1v_2$,
$v_3v_4$. Notice that the remaining three edges can only have fresh
colors. If only two colors are used, then at least two edges of $H$
have the same color. It is easy to find two vertices $v_i$, $v_j$ of
$H$, such that no rainbow tree connects $S$, where
$S=\{v'_1,v'_2,v_i,v_j\}$. Assume then only one edge of $H$ has the
color of cut edge, say $v'_1v''_1$ of $T_i$. Suppose $d(v'_1,H)\geq
d(v''_1,H)$. Then any tree connecting $v'_1$ and the three vertices
of $H$ except $v_i$ makes use of at least three edges of $H$, namely
at least three new distinct colors are needed in $H$. Thus the
result follows.
\end{proof}

\begin{figure}[h,t,b,p]
\begin{center}
\scalebox{1}[1]{\includegraphics{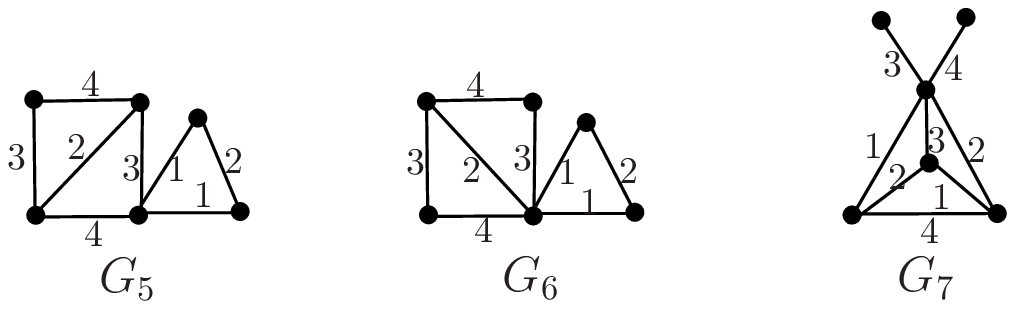}}\\[15pt]

Figure~4. Graphs for Theorem \ref{thm2}
\end{center}
\end{figure}

Now we are prepared to characterize the graphs of order $n$ whose 4-rainbow index is $n-1$.

\begin{thm}\label{thm2}
Let $G$ be a graph of order $n$. Then $rx_4(G)=n-1$ if and only if $G$ is a
tree, or a unicyclic graph, or a cactus with $c(G)=2$, or $G\in \mathcal{G}_1\cup \mathcal{G}_2$.
\end{thm}

\begin{proof}
We only need to prove the necessity. Let $G$ be a graph with
$rx_4(G)=n-1$. By Lemma \ref{theta}, we know that if $G$ is not a
tree or a unicyclic graph or a cactus with $c(G)=2$, then $G$
contains a $K_4$ or $K_4-e$ as an induced subgraph. Now suppose that
$G$ contains a $K_4$ or $K_4-e$ but $G\notin \mathcal{G}_1\cup
\mathcal{G}_2$. Consider the three graphs $G_5$, $G_6$, $G_7$. By
checking the given coloring in Figure 4, we have $rx_4(G_i)\leq
n-2$, $i=5,6,7$. Thus we can extend $G_5$, $G_6$ or $G_7$ to get a
spanning subgraph $G'$ of $G$, then $rx_4(G)\leq rx_4(G')\leq n-2$,
a contradiction.
\end{proof}

\end{document}